\documentclass[11pt]{amsart}
\usepackage{amscd, amsmath, amsthm, amssymb, mathrsfs}
\usepackage{verbatim}
\usepackage[all]{xy}
\usepackage{color}
\usepackage{hyperref}

\newtheorem{theorem}{Theorem}[section]
\newtheorem{lemma}[theorem]{Lemma}
\newtheorem{proposition}[theorem]{Proposition}

\newtheorem{corollary}[theorem]{Corollary}

\theoremstyle{definition}
\newtheorem{definition}[theorem]{Definition}
\newtheorem{example}[theorem]{Example}

\theoremstyle{remark}
\newtheorem{remark}[theorem]{Remark}

\numberwithin{equation}{section}

\newcommand\CC{{\mathbb{C}}} 
\newcommand\LL{{\mathscr{L}}} \newcommand\MM{{\mathscr{M}}} \newcommand\OO{{\mathscr{O}}} \newcommand\PP{{\mathbb{P}}} \newcommand\QQ{{\mathbb{Q}}}  \newcommand\ZZ{{\mathbb{Z}}}

\newcommand\Aut{{\text{\rm Aut}}}

\newcommand\Bl{{\text{\rm Bl}}}

\newcommand\Gr{\text{\rm Gr}}
\newcommand\Eff{\text{\rm Eff}}
\newcommand\Mov{\text{\rm Mov}}
\newcommand\Nef{\text{\rm Nef}}
\newcommand\Pic{\text{\rm Pic}}
\newcommand\PGL{\text{\rm PGL}}
\newcommand\PO{\text{\rm PO}}

\newcommand\tih{\widetilde{H}}

\newcommand\xht{X^{[2]}} \newcommand\xst{X^{(2)}}
 
\newcommand\zht{Z^{[2]}} \newcommand\zst{Z^{(2)}}
\newcommand\fht{f^{[2]}}

\begin{document}

\title[Automorphisms of Hilbert squares]{On automorphisms of Hilbert squares of smooth hypersurfaces}

\author[Long\ \ Wang]{Long\ \ Wang}
\address{Graduate School of Mathematical Sciences, the University of Tokyo, 3-8-1 Komaba, Meguro-Ku, Tokyo 153-8914, Japan}
\email{wangl11@ms.u-tokyo.ac.jp}

\keywords{Hilbert scheme of points, automorphism group, hypersurface}

\subjclass[2020]{14J50, 14C05, 14J70}

\begin{abstract} Let $X$ be a smooth projective hypersurface of dimension at least three. We show that every automorphism of the Hilbert square $\xht$ of $X$ is induced by some automorphism of $X$.
\end{abstract}

\maketitle

\thispagestyle{empty}
%%%%%%%%%%%%%%%%%%%%%%%%%%%%%%%%%%%%%%%%%%%

\setcounter{section}{0}

\section{Introduction}

Hilbert schemes of points on surfaces have been studied for many years with techniques from different branches of mathematics, such as algebraic geometry, symplectic geometry  and representation theory (\cite{Na99, Qi18}). The birational geometry and automorphisms of these Hilbert schemes was explored recently (see \cite{Hu17} for a survey). One interesting question is whether each automorphism of the Hilbert scheme is natural, that is, it is induced by some automorphism of the underlying surface. This is not true in general when the underlying surface is a K3 surface (\cite{BC16, HT16, Og16, Ca19}). In \cite{BOR20}, this question was settled when the underlying surface is either weak del Pezzo or of general type (see also \cite{Ha19}). 

It is well-known that the Hilbert scheme of points on a smooth projective variety can be rather singular if the dimension of the underlying variety is greater than two. More precisely, let $X$ be a smooth projective variety of dimension $\geq 3$ and let $X^{[n]}$ be the Hilbert scheme of $n$ points on $X$, then $X^{[n]}$ is smooth if and only if $n \leq 3$ (\cite[Page 40]{Ch98}). Recently, Hilbert squares, i.e., Hilbert schemes  of two points have been studied from different aspects (\cite{GS14, Vo17, BFR19, Mb19, Sh19, BOR20, BFR20}). In this paper, we concentrate on their automorphism groups. Recall that every automorphism of $(\PP^n)^{[2]}$ is natural by \cite[Theorem 4]{BOR20}. Our aim is to generalise this result to the Hilbert square of a smooth projective hypersurface. This answers a question in \cite[Remark 21]{BFR20}.

\begin{theorem}\label{main} Let $N\geq 4$ be a positive integer. Let $X\subset \PP^N$ be a smooth projective hypersurface of degree $d \geq 2$. Then every automorphism of $\xht$ is natural.
\end{theorem}

Combined with the main results of \cite{MM64}, we obtain the following.

\begin{corollary}\label{cor} Let $N\geq 4$ be a positive integer. Let $X\subset \PP^N$ be a smooth projective hypersurface of degree $d \geq 3$. Then the automorphism group $\Aut(\xht)$ of $\xht$ is finite. If in addition $X$ is general\footnote{``General'' means lying in some Zariski-dense open subset of the parameter space.}, then $\Aut(\xht)$ is trivial.
\end{corollary}

In order to achieve Theorem \ref{main}, it is enough to show that $\Aut(\xht)$ acts on the nef cone $\Nef(\xht)$ of $\xht$ identically by general strategy as in \cite{BOR20, Ha19}. We will study the extremal contractions induced by two rays of the boundary of $\Nef(\xht)$. %the nef cone of $\xht$. 
Clearly one contraction is precisely the Hilbert--Chow morphism $\xht \to \xst$, where $\xst$ denotes the second symmetric product of $X$. We first handle the case where $K_X$ is nef (see Lemma \ref{nef}). For the case where $X$ is Fano, we will construct the second contraction using the lines on hypersurfaces and investigate the property of this contraction (see Lemmas \ref{type} and \ref{fano}).

After some preliminaries in Section \ref{prel} and reductions in Section \ref{reduc}, we will prove Theorem \ref{main} in Section \ref{prof}. Some examples are given in Section \ref{exam}.

\medskip\noindent \textbf{Acknowledgements.} I thank Professor J{\o}rgen Vold Rennemo for suggesting Lemma \ref{nef} and an improvement of Lemma \ref{fano} which strengthens the main result of this paper. I also thank Professor Keiji Oguiso for suggestions, Professor Klaus Hulek for his help, and the referee for comments. This work is partially supported by Leading Graduate Course for Frontiers of Mathematical Sciences and Physics.

%\medskip

\section{Preliminaries}\label{prel}

We work over the field $\mathbb{C}$ of complex numbers. Let $X$ be a smooth projective variety of dimension $n$. We denote by $\Aut(X)$ the automorphism group of $X$ and by $\Pic(X)$ the Picard group of $X$. We denote by $\Nef(X)$, $\Mov(X)$, and $\Eff(X)$ the nef cone, the movable cone, and the effective cone of $X$ which are generated by numerical classes of nef divisors, movable divisors and effective divisors on $X$, respectively. Recall that a divisor is called \textit{movable}, if its stable base locus has codimension at least two. We refer to \cite{De01, Ma02} for the knowledge of birational geometry.

\subsection{} We denote by $\xht$ the Hilbert square of $X$, i.e., the Hilbert scheme of zero dimensional subschemes of length two of $X$. Then $\xht$ is a smooth projective variety of dimension $2n$. One can construct $\xht$ as follows. Let $\Gamma \subset X \times X$ be the diagonal and $\Bl_{\Gamma}(X\times X)$ the blowup of $X\times X$ along $\Gamma$. The natural involution $\theta$ on $X \times X$ lifts to the blowup and the quotient is canonically isomorphic to the Hilbert square $\xht$ of $X$. One can also regard $\xht$ as the blowup of $(X \times X)/\theta$ along the image of the diagonal $\Gamma$ in the second symmetric product $\xst:= (X \times X)/\theta$ of $X$. The morphism $\xht \to \xst$ is usually called the \textit{Hilbert--Chow morphism}, and we denote by $B$ its exceptional divisor.

Assume that $h^1(X, \OO_X) = 0$, then we have the following isomorphism
\[ \Pic(\xht) \cong \Pic(X) \oplus \ZZ\cdot \frac{B}{2}. \]
Here we note that $B/2$ is a divisor on $\xht$. For a divisor $H \in \Pic(X)$, denote by $\widetilde{H} \in \Pic(\xht)$ the image of $H$ under this isomorphism. Then we have the following formula about canonical divisors:
\[ K_{\xht} = \widetilde{K_X} + \frac{n - 2}{2} B. \]
In fact, let $E$ be the exceptional divisor of the blowup $p: \Bl_{\Gamma}(X\times X) \to X \times X$, and let $q: \Bl_{\Gamma}(X\times X) \to \xht$ be the quotient map. By \cite[(97) and (98) in Page 22]{BFR20}, we have 
\[ K_{\Bl_{\Gamma}(X\times X)} = q^{\ast}(K_{\xht} + B/2) \ \text{ and } \ q^{\ast}(B/2) = E. 
\]
We write $K_{\xht} = \widetilde{H} + a B/2$ with $H \in \Pic(X)$. Then $K_{\Bl_{\Gamma}(X\times X)} = q^{\ast} \big( \widetilde{H} + (a+1) B/2 \big)$. On the other hand, we have $K_{\Bl_{\Gamma}(X\times X)} = p^{\ast}K_{X\times X} + (n - 1) E$. It is then not difficult to see that $H = K_X$ and $a = n - 2$. 

An automorphism $f \in \Aut(X)$ induces an automorphism $\fht \in \Aut(\xht)$. In particular, there is a natural inclusion $\Aut(X) \subset \Aut(\xht)$. An automorphism $\sigma \in \Aut(\xht)$ is called \textit{natural}, if $\sigma = \fht$ for some automorphism $f \in \Aut(X)$.

\subsection{} Let us collect some facts about lines on hypersurfaces. Let $X \subset \PP^N$ be a projective variety. We denote by $F(X)$ the Hilbert scheme of lines on $X$, which is a subscheme of the Grassmannian of lines in $\PP^N$ and is called the \textit{Fano scheme of lines} on $X$. We refer to \cite[Theorem 8]{BV78} and \cite[Theorem V.4.3]{Ko96} for the following result (see also \cite{AK77, De01, EH16}).

\begin{theorem}\label{basic} Let $X \subset \PP^N$ be a smooth projective hypersurface of degree $d$. Then the following hold.

\medskip $\mathrm{(1)}$ $F(X)$ is non-empty if $d \leq 2N - 3$.

\medskip $\mathrm{(2)}$ $F(X)$ is empty for general $X$ if $d > 2N - 3$.

\medskip $\mathrm{(3)}$ $F(X)$ is smooth of dimension $2N - 3 - d$ for general $X$ if $d \leq 2N - 3$. 
\end{theorem}

\begin{example} The Fermat hypersurface $X^d_N$ is the hypersurface in $\PP^N$ defined by the equation
\[ x^d_0 + x^d_1 + \cdots + x^d_N = 0. \]
It is smooth and contains lines. Indeed, it contains a $(N-3)$-dimensional family of lines. See for instance, \cite[2.14]{De01}. Therefore, one can see that the generality condition in Theorem \ref{basic} is indispensable.
%By contrast, a general hypersurface in $\PP^N$ of degree $d > 2N - 3$ contains no lines (\cite[V.4.3]{Ko96}).
\end{example}

The conjecture due to Debarre and de Jong predicts that $F(X)$ always has the expected dimension $2N - d - 3$ for an arbitrary smooth Fano hypersurface $X \subset \PP^N$ (see e.g., \cite[Conjecture 6.39]{EH16}). This conjecture is open in general whereas the following partial result is known.

\begin{theorem}[{\cite[Corollary 3.3]{Be14}}]\label{bemain} Let $X \subset \PP^N$ be a smooth Fano hypersurface of degree $d \leq 8$. Then $F(X)$ has the expected dimension $2N - d - 3$.
\end{theorem}

%\medskip

\section{Reductions}\label{reduc}

In this section, we reduce Theorem \ref{main} to Proposition \ref{propred}. The proof is identical with \cite[Section 3]{BOR20} (also \cite{Ha19}) where the case of Hilbert schemes of points on surfaces is considered. 

\begin{lemma}\label{bor6} Let $f: X \to Y$ be a morphism of projective varieties, where $Y$ is normal and $f$ has connected fibers. Let $\LL$ be an ample line bundle on $Y$. Then for any automorphism $\sigma: X \to X$ such that $\sigma^{\ast}f^{\ast}\LL = f^{\ast}\LL$, there exists an isomorphism $\tau: Y \to Y$ such that $\tau \circ f = f \circ \sigma$.
\end{lemma}

\begin{proof} See \cite[Lemma 6]{BOR20}. 
\end{proof}

Similar results as the following first appeared in \cite[Theorem 1]{BS12} and later in \cite[Theorem 1.4]{Ha17}.

\begin{proposition}\label{propred} Let $X$ be a simply connected smooth projective variety of dimension $\geq 2$ and Picard number $1$. Assume that $\Pic(X) = \ZZ \cdot H$, where $H$ is an ample divisor. Let $\sigma \in \Aut(\xht)$ be an automorphism satisfying $\sigma^{\ast}\tih = \tih$ and $\sigma^{\ast}B = B$. Then $\sigma$ is natural.
\end{proposition}

\begin{proof} Since $\sigma^{\ast}\tih = \tih$, by Lemma \ref{bor6}, there exists an automorphism $\tau : \xst \to \xst$ which fits the following commutative diagram:
\[
\CD
  \xht @>\sigma>> \xht \\
  @V \mathrm{HC} VV @V \mathrm{HC} VV  \\
  \xst @>\tau>> \xst.
\endCD
\]

Since $\tau$ preserves the singular points, $\tau$ preserves the diagonal $\Delta \subset \xst$. Let $\Gamma \subset X^2 := X \times X$ be the diagonal and consider the restriction of the quotient map $\rho : X^2 \setminus \Gamma \to \xst \setminus \Delta$. Since $X$ is simply connected and $\Gamma \subset X^2$ is of codimension at least $2$, $X^2 \setminus \Gamma$ is also simply connected. Hence $\rho$ is the universal covering space of $\xst \setminus \Delta$. Applying the universal lifting property to the morphism $\tau \circ \rho$, we obtain an automorphism $f^{\prime} \in \Aut(X^2 \setminus \Gamma)$ which fits the following commutative diagram:
\[ \CD
  X^2 \setminus \Gamma @>f^{\prime}>> X^2 \setminus \Gamma \\
  @V \rho VV @V \rho VV  \\
  \xst \setminus \Delta @>\tau>> \xst \setminus \Delta.
\endCD
\]
By the argument of \cite[Proposition 12]{BOR20}, $f^{\prime}$ extends to an automorphism $f: X^2 \to X^2$ such that $\tau \circ \rho = \rho \circ f$.

Let $\LL$ be a very ample line bundle on $X$, and let $\LL_i = p_i^{\ast}\LL$ be its pull-back to $X^2$, where $p_i: X \times X \to X$ are the natural projections. Clearly the projection $p_i$ is naturally identified with the morphism associated with the globally generated line bundle $\LL_i$. We consider $f^{\ast}\LL_i$ following the argument of \cite[Theorem 4.1(d)]{Og16}. Since $h^1(X, \OO_X) = 0$, we know that $\Pic(X^2) = \Pic(X) \bigoplus \Pic(X)$. In particular, there exist line bundles $\MM_{j}$ on $X$ such that $f^{\ast}\LL_i \cong p_1^{\ast}\MM_1 \otimes p_2^{\ast}\MM_2$. Moreover, these $\MM_{j}$ are globally generated.

Let $g_j: X \to Y_j$ be the morphism associated with $\MM_j$. Then $p_i \circ f = h \circ (g_1 \times g_2)$ for some isomorphism $h: Y_1\times Y_2 \to X$. Since $X$ is of Picard number $1$, we know that one of $Y_j$ is a point and the other is isomorphic to $X$. Therefore, $f = f_1 \times f_2$ for some $f_i \in \Aut(X)$. Since $\tau \in \Aut(\xst)$ preserves the diagonal $\Delta \subset \xst$, we know that $f \in \Aut(X^2)$ preserves the diagonal $\Gamma \subset X^2$. In particular, $f_1 = f_2$. This implies that $\sigma$ is natural.
\end{proof}

%\medskip

\section{Proof of Theorem \ref{main}}\label{prof}

Let $N \geq 4$ be a positive integer. Let $X \subset \PP^N$ be a smooth projective hypersurface. Then $\dim X = N - 1 \geq 3$ and $\Pic(X) \cong \Pic(\PP^N)$ by the Lefschetz hyperplane theorem. Thus, the Picard number $\rho(\xht) = 2$. Let us consider the nef cone of $\xht$, whose boundary consists of two rays. Let $H$ be a very ample divisor of $X$ which gives the embedding $X \subset \PP^N$. Then $\widetilde{H}$ is a globally generated and big divisor on $\xht$ which is not ample, and $\widetilde{H}$ induces the divisorial contraction $\xht \to \xst$, the Hilbert--Chow morphism. In particular, $\widetilde{H}$ is on the boundary of the nef cone $\Nef(\xht)$ of $\xht$.

Let us first deal with the case where $\deg X \geq N+1$, or equivalently, $K_X$ is nef. We have the following slightly more general result.

\begin{lemma}\label{nef} Let $Z$ be a smooth projective variety of dimension $n \geq 3$ with $h^1(Z, \OO_Z) = 0$ and $\Pic(Z) = \ZZ\cdot L$ for an ample divisor $L$. If $K_Z$ is nef, then every automorphism of $\zht$ acts as the identity on $\Pic(\zht)$.
\end{lemma}

\begin{proof} Let $f : \zht \to \zht$ be an automorphism. Suppose that $f^{\ast}$ acts non-trivially on $\Pic(\zht)$. Then $f^{\ast}$ must exchange the two extremal rays of $\Nef(\zht)$. Note also that
\[ \Pic(\zht) = \ZZ\cdot\widetilde{L} \oplus \ZZ\cdot \frac{D}{2}, \]
where $D$ is the exceptional divisor of the Hilbert--Chow morphism $\zht \to \zst$. Since $\widetilde{L}$ is nef, primitive and extremal, $f^{\ast}(\widetilde{L})$ is the primitive generator of the other extremal ray of $\Nef(\zht)$. It follows that $f^{\ast}(f^{\ast}(\widetilde{L})) = \widetilde{L}$, and $f^{\ast}\circ f^{\ast} = \mathrm{id}_{\Pic(\zht)}$. Moreover, $\widetilde{L} + f^{\ast}(\widetilde{L})$ is ample as it is lying in the interior of the nef cone.

We have $f^{\ast}(K_{\zht}) = K_{\zht}$, and
\[ K_{\zht} = \widetilde{K_Z} + \frac{n - 2}{2}D \]
is non-zero (because $n \geq 3$). Since $f^{\ast}$ acts non-trivially, the $f$-fixed part of $\Pic(\zht)_{\QQ}$ is spanned by $K_{\zht}$. As $\widetilde{L} + f^{\ast}(\widetilde{L})$ is fixed by $f^{\ast}$ as well, we have
\[ \widetilde{L} + f^{\ast}(\widetilde{L}) = \alpha K_{\zht} \]
for some non-zero $\alpha \in \QQ$. So either $K_{\zht}$ or $-K_{\zht}$ is ample.

One can find that $K_{\zht}$ is never ample, since any curve in a fiber of the Hilbert--Chow morphism intersects $K_{\zht}$ negatively. Hence $-K_{\zht}$ must be ample. For each curve $C \subset Z$ and a fixed point $z \in Z \setminus C$, there is a curve $\widetilde{C}$ on $\zht$ of the form $\{(z, y) \in \zht \,|\, y \in C\}$. Since $-K_{\zht}\cdot \widetilde{C} > 0$, we have $-K_Z\cdot C > 0$. This implies that $-K_Z$ is ample, which  is a contradiction.
\end{proof}

In what follows, we handle the case where $X \subset \PP^N$ is a smooth Fano hypersurface, that is, $\deg X \leq N$. Let us seek for the second ray of $\partial\Nef(\xht)$ and the associated contraction, which we call the \textit{second contraction} of $\xht$ for convenience. The remaining question is to construct and analyze the second contraction.

\begin{example}[{\cite[Theorem 4]{BOR20}}]\label{exps} Let us consider the case $X = \PP^N$. Then the Picard number of $(\PP^N)^{[2]}$ is two, and the two contractions are exactly the Hilbert--Chow morphism
\[(\PP^N)^{[2]} \to (\PP^N)^{(2)} \]
and the morphism
\[ (\PP^N)^{[2]} \to \Gr(2, N+1) \]
which maps a length-two subscheme of $\PP^N$ to the line passing through it.

Since $\dim (\PP^N)^{[2]} = 2N$ while $\dim \Gr(2, N+1) = 2(N - 1)$, the second contraction of $(\PP^N)^{[2]}$ is of fiber type. This implies that each automorphism $\sigma \in \Aut((\PP^N)^{[2]})$ cannot exchange two rays of the boundary of $\Nef((\PP^N)^{[2]})$ and hence must preserve them up to scaling. Since $\sigma$ also preserves the divisibility, we can see that $\sigma$ acts on $\Nef((\PP^N)^{[2]})$ identically. Therefore, by Proposition \ref{propred}, every automorphism of $(\PP^N)^{[2]}$ is natural.
\end{example}

In what follows, we assume that the smooth hypersurface $X \subset \PP^N$ is of degree $\geq 2$ (and of dimension $\geq 3$). Similar to Example \ref{exps}, we have the following morphism
\[ \psi: \xht \to \Gr(2, N+1) \]
that sends $z \in \xht$ to the line $\ell_z$ generated by $z$ in $\PP^N$. Clearly, we have the following fact.

\begin{lemma} The morphism $\psi$ is finite if and only if $X$ contains no lines. In other words, $\psi$ contracts curves if and only if $X$ contains a line.
\end{lemma}

Assume that $X$ contains a line. We take the Stein factorization of $\psi$:
\[ \xht \xrightarrow{\phi} W \rightarrow \Gr(2, N+1), \]
where $W \to \Gr(2, N+1)$ is a finite morphism and
\[ \phi: \xht \to W \]
admits connected fibers. Note that $\dim \xht = 2N - 2 = \dim \Gr(2, N+1)$. We will see that the birational morphism $\phi$ is precisely the second contraction of $\xht$. %We can determine the type of this contraction.

\begin{lemma}\label{type} Let $N\geq 4$ be a positive integer. Let $X\subset \PP^N$ be a smooth Fano hypersurface of degree $d \geq 2$.

\medskip $(1)$ If $d = 2$, then the contraction $\phi$ is divisorial.

\medskip $(2)$ If $3 \leq d \leq 8$, then the contraction $\phi$ is small.
\end{lemma}

\begin{proof} Notice that $X$ contains a line by Theorem \ref{basic} (1). The map $\phi$ contracts the subset
\[ T:= \{ z \in \xht \,:\, \mathrm{the \ line} \ \ell_z \text{ is contained in } X \} \]
of $\xht$. One can see that each fiber of $T \to F(X)$ is $(\PP^1)^{[2]} \cong \PP^2$, and hence $\dim T = \dim F(X) + 2$.

If $d = 2$, then $\dim F(X) = 2N - 5$ by Theorem \ref{bemain}, so
\[ \dim T = 2N - 3 = \dim \xht - 1. \]
Therefore, $\phi$ is divisorial.

If $3 \leq d \leq 8$, then by Theorem \ref{bemain},
\[ \dim T = 2N - d - 1 \leq 2N - 4 = \dim \xht - 2. \]
Hence we can conclude that $\phi$ is small.
\end{proof}

\begin{remark} $(1)$ In the proof of Theorem \ref{main}, we will only use Lemma \ref{type} for the case $N = 4$ and $2 \leq d \leq 4$. In Section \ref{exam}, we will discuss the Hilbert square $\xht$ with $X \subset \PP^N$ of degree $d \leq 4$ in more detail, where $N \geq 4$ is arbitrary.

\medskip $(2)$ As we mentioned in Section \ref{prel}, the conjecture of Debarre and de Jong says that $F(X)$ always has the expected dimension $2N - d - 3$ for an arbitrary smooth Fano hypersurface $X \subset \PP^N$. This would imply that the contraction $\phi: \xht \to W$ is small for all $3 \leq d \leq N$, which is true for general $X$ by Theorem \ref{basic} (3).
\end{remark}

\begin{lemma}\label{fano} Let $X \subset \PP^N$ be a smooth Fano hypersurface with $N \geq 4$, then all automorphisms of $\xht$ act as the identity on $\Pic(\xht)$.
\end{lemma}

\begin{proof} By Theorem \ref{basic} (1), $X$ contains a line. Thus,
\[ \phi: \xht \to W \]
is a contraction of $\xht$. The fibers of
\[ \psi: \xht \to \Gr(2, N + 1) \]
are either points or isomorphic to $\PP^2$; while the fibers of the Hilbert--Chow morphism $\xht \to \xst$ are either points or isomorphic to $\PP^{N-2}$.

If $N \geq 5$, then these two contractions are non-isomorphic. So we know that every automorphism of $\xht$ cannot interchange two rays of the boundary of $\Nef(\xht)$ and must preserve them up to scaling. Since every automorphism also preserves the divisibility, one can conclude that it acts on $\Nef(\xht)$ identically.

Let us consider the remaining cases, namely, $N = 4$ and $2 \leq d \leq 4$. The argument is indeed available for arbitrary $N \geq 4$. Assume that $d = 2$. Consider the morphism
\[ \psi: \xht \to \Gr(2, N+1). \]
Each generic line in $\PP^N$ intersects the hyperquadric $X$ in $2$ points (with multiplicity), so the map $\psi$ is generically one-to-one. This implies that $\psi$ itself is a birational contraction. By Lemma \ref{type} (1), we know that $\psi$ is divisorial. The other contraction of $\xht$ is the Hilbert--Chow morphism $\xht \to \xst$. These two contractions are non-isomorphic, since they have non-isomorphic images. Thus, every automorphism of $\xht$ fixes two rays of the boundary of $\Nef(\xht)$, which implies that it acts on $\Nef(\xht)$ identically.

Now assume that $3 \leq d \leq 4$. Notice that every automorphism of $\xht$ preserves the movable cone $\Mov(\xht)$ of $\xht$ as well. Since the first ray of $\Nef(\xht)$ induces the Hilbert--Chow morphism which is divisorial, one can see that it lies on the boundary of $\Mov(\xht)$. By Lemma \ref{type} (2), the second ray induces a small contraction of $\xht$, so it lies in the interior of $\Mov(\xht)$. This implies that every automorphism of $\xht$ preserves two rays of the boundary of $\Nef(\xht)$ and hence it acts on $\Nef(\xht)$ identically.
\end{proof}

In summary, we have the following proposition (by Lemmas \ref{nef} and \ref{fano}).

\begin{proposition}\label{mainprop} Let $X \subset \PP^N$ be a smooth projective hypersurface with $N \geq 4$, then all automorphisms of $\xht$ act as the identity on $\Pic(\xht)$. \qed
\end{proposition}

\begin{proof}[Proof of Theorem \ref{main}] This follows immediately from Proposition \ref{propred} and Proposition \ref{mainprop}.
\end{proof}

\begin{proof}[Proof of Corollary \ref{cor}] By \cite[Page 347]{MM64}, the automorphism group of a smooth projective hypersurface $X$ of dimension $\geq 3$ and degree $d \geq 3$ is finite, and is trivial provided additionally $X$ general. The assertion then follows from Theorem \ref{main}.
\end{proof}

\begin{remark} One can also see that $\Aut((\PP^N)^{[2]}) \cong \Aut(\PP^N)$ is the projective linear group $\PGL(N+1, \CC)$ when $N \geq 2$. For a smooth hyperquadric $Q \subset \PP^N$ of dimension $\geq 3$, $\Aut(Q^{[2]}) \cong \Aut(Q)$ is the projective orthogonal group $\PO(N+1, Q)$, where $Q$ is regarded as a quadratic form in $N+1$ (complex) variables.
\end{remark}

\begin{remark} Consider a simply connected smooth projective variety $Z$ of Picard number $1$ and of dimension $\geq 3$ whose canonical divisor $K_Z$ is nef. Then by Lemma \ref{nef} and Proposition \ref{propred}, every automorphism of $\zht$ is natural. A typical example of such $Z$ is a general complete intersection of dimension $\geq 3$ in $\PP^N$ which is either Calabi--Yau or of general type. It would be interesting to investigate the case of a Fano complete intersection.
\end{remark}

%\medskip

\section{Examples}\label{exam}

In this section, we discuss the case where the underlying hypersurface is of degree $\leq 4$ in more detail. Let $N \geq 4$ be a positive integer. Let $X \subset \PP^N$ be a smooth projective hypersurface of degree $d\leq 4$.

Let $Z$ be a normal $\QQ$-factorial projective variety. Recall that a \textit{small $\QQ$-factorial modification} of $Z$ is a birational map $Z \dashrightarrow W$ to another normal $\QQ$-factorial projective variety $W$ which is isomorphic in codimension $1$. The basic example of a small $\QQ$-factorial modification is a flip.

\begin{example} When $d = 1$, $X$ is isomorphic to $\PP^{N-1}$. By Example \ref{exps}, we know that the second contraction of $\xht$ is of fiber type. Thus, the only small $\QQ$-factorial modification of $\xht$ is the identity map.
\end{example}

\begin{example} When $d = 2$, $X$ is a smooth hyperquadric. By Lemma \ref{type} (1), we know that the second contraction of $\xht$ is divisorial. Again the only small $\QQ$-factorial modification of $\xht$ is the identity map.
\end{example}

\begin{example} When $d = 3$, $X$ is a smooth cubic hypersurface. By Lemma \ref{type} (2), we know that the second contraction of $\xht$ is small. Thus, there exists a non-trivial small $\QQ$-factorial modification of $\xht$. One can describe it explicitly as in \cite[Section 4]{BFR20}.

Let $T_{\PP^N}|_X$ be the restriction of the tangent bundle of $\PP^N$ to $X$, and let $P_X := \PP(T_{\PP^N}|_X) \to X$ be its projectivization. Then $P_X$ is a $\PP^{N-1}$-bundle over $X$, and parametrises a point on $X$ together with a line in $\PP^N$ passing through it. There is a map $\zeta: \xht \dashrightarrow P_X$ sending $z \in \xht$ to $(\ell_z, x) \in P_X$, where $\ell_z$ is the line determined by $z$, and $x$ is the residual intersection point of $\ell_z$ with $X$ (\cite[(5.3)]{GS14}). Based on \cite{GS14, Vo17}, in \cite[Corollary 15]{BFR20}, it is shown that $\zeta$ is a standard flip. Since the natural morphism $P_X := \PP(T_{\PP^N}|_X) \to X$ is of fiber type, one can see that $\zeta: \xht \dashrightarrow P_X$ is the unique non-trivial small $\QQ$-factorial modification of $\xht$.
\end{example}

As it is shown, when $d = 1, 2$, or $3$, the Hilbert square $\xht$ is a smooth Fano variety (\cite[Lemma 2.9]{Sa20}, \cite[Proposition 20]{BFR20}). In particular, $\xht$ is a Mori dream space (\cite[Corollary 1.3.2]{BCHM10}) in the following sense of Hu--Keel \cite[Definition 1.10]{HK00}.

\begin{definition} Let $X$ be a normal $\QQ$-factorial projective variety. We call $X$ a \textit{Mori dream space}, if the following three conditions are satisfied:

\medskip $(1)$ $\Pic(X)_{\QQ} = N^1(X)_{\QQ}$, or equivalently, $h^1(X, \OO_X) = 0$.

\medskip $(2)$ $\Nef(X)$ is generated by finitely many semi-ample divisors as a convex cone.

\medskip $(3)$ There are finitely many small $\QQ$-factorial modifications $f_i: X \dashrightarrow X_i$, $1 \leq i \leq n$, such that each $X_i$ satisfies $(1)$ and $(2)$ and
\[ \Mov(X) = \bigcup^n_{i = 1}f^{\ast}_i(\Nef(X_i)). \]
\end{definition}

For more details about Mori dream spaces, we refer to \cite{HK00, McK10, Ca18}. In what follows, we consider the Hilbert square of a smooth quartic hypersurface (of dimension $\geq 3$), on which we will construct a birational involution. This is analogous to the classical Beauville involution \cite[Section 6]{Be83} (see also \cite[6.1]{BC16}).

\begin{example}[The Beauville involution] Let $S \subset \PP^3$ be a smooth quartic surface with Picard number $1$. Then $S$ is a K3 surface. Consider the morphism
\[ \psi: S^{[2]} \to \Gr(2, 4), \]
that maps $z \in S^{[2]}$ to the line $\ell_z$ in $\PP^3$ spanned by $z$. Since $S$ has Picard number $1$, it contains no lines. In particular, the morphism $\psi$ is finite. Note that each line in $\PP^3$ intersects $S$ in $4$ points (with multiplicity). One can define a rational map
\[ \iota: S^{[2]} \dashrightarrow S^{[2]} \]
that sends $z$ to the length-two subscheme $z^{\prime}$ defined by $\ell_z \cap S = z \amalg z^{\prime}$.

Note that $\iota$ is defined at a point $z \in S^{[2]}$ if and only if the line $\ell_z$ is not contained in $S$. Thus, the map $\iota$ is indeed everywhere defined again due to the fact that $S$ contains no lines. Therefore, we obtain an involution $\iota \in \Aut(S^{[2]})$ which is not natural and fits the following commutative diagram
\[ \xymatrix{
  S^{[2]} \ar[rr]^{\iota} \ar[dr]_{\psi}
                &  &    S^{[2]} \ar[dl]^{\psi}    \\
                & \Gr(2, 4)                 }.
\] 
\end{example}

\begin{example} Let $X \subset \PP^N$ be a smooth quartic hypersurface of dimension $\geq 3$. Similar to the Beauville involution, we can define the following involution
\[ \iota: \xht \dashrightarrow \xht \]
by sending $z \in \xht$ to the length-two subscheme $z^{\prime}$ defined by $\ell_z \cap X = z\amalg z^{\prime}$, where $\ell_z$ denotes the line associated with $z$. Since $X$ must contain lines in this case by Theorem \ref{basic} (1), we know that $\iota$ can never be an automorphism of $\xht$.

Since $\dim F(X) = 2N - 7 \geq 1$ while $\dim \xht = 2N - 2$, by the proof of Lemma \ref{type} (2), we know that $\iota$ is defined outside a subvariety of $\xht$ of codimension $3$. Hence $\iota$ is a small birational automorphism of $\xht$ which fits the following commutative diagram
\[ \xymatrix{
  \xht \ar@{-->}[rr]^{\iota} \ar[dr]_{\psi}
                &  &    \xht \ar[dl]^{\psi}    \\
                & \Gr(2, N+1)                 }.
\]
We can also conclude that $\iota: \xht \dashrightarrow \xht$ is the only non-trivial small $\QQ$-factorial modification of $\xht$ and that $\xht$ is a Mori dream space.
\end{example}


\medskip

\begin{thebibliography}{aCZ}


\bibitem{AK77} Altman, A., Kleiman, S. (1977). Foundations of the theory of Fano schemes. \textit{Comp. Math.} 34(1):3--47.


\bibitem{Be83} Beauville, A. (1983). Some remarks on K\"{a}hler manifolds with $c_1 = 0$. \textit{Classification of algebraic and analytic manifolds}. Progr. Math., Vol. 39. Boston: Birkh\"{a}user, pp. 1--26. 


\bibitem{Be14} Beheshti, R. (2014). Hypersurfaces with too many rational curves. \textit{Math. Ann.} 360(3):753--768. DOI: 10.1007/s00208-014-1024-8.


\bibitem{BCHM10} Birkar, C., Cascini, P., Hacon, C. D., McKernan, J. (2010). Existence of minimal models for varieties of log general type. \textit{J. Amer. Math. Soc.} 23(2):405--468. DOI: 10.1090/S0894-0347-09-00649-3. 


\bibitem{BC16} Boissi\`{e}re, S., Cattaneo, A., Nieper-Wisskirchen, M., Sarti, A. (2016). The automorphism group of the Hilbert scheme of two points on a generic projective K3 surface. \textit{K3 surfaces and their moduli}. Progr. Math., Vol. 315. Cham: Birkh\"{a}user, pp. 1--15. DOI: 10.1007/978-3-319-29959-4\_1. 


\bibitem{BFR19} Belmans, P., Fu, L., Raedschelders, T. (2019). Hilbert squares: derived categories and deformations. \textit{Sel. Math. New Ser.} 25(3):1--32. DOI: 10.1007/s00029-019-0482-y. 


\bibitem{BFR20} Belmans, P., Fu, L., Raedschelders, T. (2020). Derived categories of flips and cubic hypersurfaces. \textit{arXiv}:2002.04940. DOI: 10.48550/arXiv.2002.04940. 



\bibitem{BOR20} Belmans, P., Oberdieck, G., Rennemo, J. V. (2020). Automorphisms of Hilbert schemes of points on surfaces. \textit{Trans. Amer. Math. Soc.} 373(9):6139--6156. DOI: 10.1090/tran/8106. 


\bibitem{BS12} Boissi\`{e}re S., Sarti, A. (2012). A note on automorphisms and birational transformations of holomorphic symplectic manifolds. \textit{Proc. Amer. Math. Soc.} 140(12):4053--4062. DOI: 10.1090/S0002-9939-2012-11277-8. 


\bibitem{BV78} Barth, W., Van de Ven, A. (1978). Fano-varieties of lines on hypersurfaces. \textit{Arch. Math} 31(1):96--104. DOI: 10.1007/BF01226420. 


\bibitem{Ca18} Castravet, A.-M. (2018). Mori dream spaces and blow-ups. \textit{Algebraic geometry: Salt Lake City 2015}. Proc. Symp. in Pure Math., Vol. 97.1. Providence, RI: American Mathematical Society, pp. 143--168. DOI: 10.1090/pspum/097.1/01671. 


\bibitem{Ca19} Cattaneo, A. (2019). Automorphisms of Hilbert schemes of points on a generic projective K3 surface. \textit{Math. Nachr.} 292(10):2137--2152. DOI: 10.1002/mana.201800557. 


\bibitem{Ch98} Cheah, J. (1998). Cellular decompositions for nested Hilbert schemes of points. \textit{Pacific J. Math.} 183(1):39--90. DOI: 10.2140/pjm.1998.183.39. 


\bibitem{De01} Debarre, O. (2001). \textit{Higher-dimensional algebraic geometry}. Universitext. New York: Springer-Verlag. DOI: 10.1007/978-1-4757-5406-3. 


\bibitem{EH16} Eisenbud, D., Harris, J. (2016). \textit{3264 and all that: A second course in algebraic geometry}. Cambridge: Cambridge University Press. DOI: 10.1017/CBO9781139062046. 


\bibitem{GS14} Galkin, S., Shinder, E. (2014). The Fano variety of lines and rationality problem for a cubic hypersurface. \textit{arXiv}:1405.5154. DOI: 10.48550/arXiv.1405.5154. 


\bibitem{Ha17} Hayashi, T. (2017). Universal covering Calabi--Yau manifolds of the Hilbert schemes of $n$ points of Enriques surfaces. \textit{Asian J. Math.} 21(6):1099--1120. DOI: 10.4310/AJM.2017.v21.n6.a4. 


\bibitem{Ha19} Hayashi, T. (2019). Automorphisms of the Hilbert schemes of $n$ points of a rational surface and the anticanonical Iitaka dimension. \textit{Geom. Dedicata} 207(1):395--407. DOI: 10.1007/s10711-019-00504-7. 


\bibitem{HK00} Hu, Y., Keel, S. (2000). Mori dream spaces and GIT. \textit{Mich. Math. J.} 48(1):331--348. DOI: 10.1307/mmj/1030132722. 


\bibitem{HT16} Hassett B., Tschinkel, Y. (2016). Extremal rays and automorphisms of holomorphic symplectic varieties. \textit{K3 surfaces and their moduli}. Progr. Math., Vol. 315. Cham: Birkh\"{a}user, pp. 73--95. DOI: 10.1007/978-3-319-29959-4\_4.


\bibitem{Hu17} Huizenga, J. (2017). Birational geometry of moduli spaces of sheaves and Bridgeland stability. \textit{Surveys on recent developments in algebraic geometry}. Proc. Symp. Pure Math., Vol. 95. Providence, RI: American Mathematical Society, pp.101--148. DOI: 10.1090/pspum/095/01639. 


\bibitem{Ko96} Koll\'{a}r, J. (2016). \textit{Rational curves on algebraic varieties}. Berlin: Springer-Verlag. DOI: 10.1007/978-3-662-03276-3. 


\bibitem{Ma02} Matsuki, K. (2002). \textit{Introduction to the Mori program}. Universitext. New York: Springer-Verlag. DOI: 10.1007/978-1-4757-5602-9. 


\bibitem{Mb19} Mboro, R. (2019). Remarks on the $CH_2$ of cubic hypersurfaces. \textit{Geom. Dedicata} 200(1):1--25. DOI: 10.1007/s10711-018-0355-0. 


\bibitem{McK10} McKernan, J. (2010). Mori dream spaces. \textit{Jpn. J. Math.} 5(1):127--151. DOI: 10.1007/s11537-010-0944-7. 


\bibitem{MM64} Matsumura H., Monsky, P. (1963). On the automorphisms of hypersurfaces. \textit{J. Math. Kyoto Univ.} 3(3):347--361. DOI: 10.1215/kjm/1250524785. 


\bibitem{Na99} Nakajima, H. (1999). \textit{Lectures on Hilbert schemes of points on surfaces}. University Lecture Series, Vol. 18. Providence, RI: American Mathematical Society. DOI: 10.1090/ulect/018. 


\bibitem{Og16} Oguiso, K. (2016). On automorphisms of the punctual Hilbert schemes of K3 surfaces.  \textit{European Journal of Mathematics} 2(1):246--261. DOI: 10.1007/s40879-015-0070-4. 


\bibitem{Qi18} Qin, Z. (2018). \textit{Hilbert schemes of points and infinite dimensional Lie algebras}.  Mathematical Surveys and Monographs, Vol. 228. Providence, RI: American Mathematical Society. DOI: 10.1090/surv/228. 


\bibitem{Sa20} Sawin, W. (2020). Freeness alone is insufficient for Manin-Peyre. \textit{arXiv}:2001.06078. DOI: 10.48550/arXiv.2001.06078. 


\bibitem{Sh19} Shen, M. (2019). Rationality, universal generation and the integral Hodge conjecture. \textit{Geom. Topol.} 23(6):2861--2898. DOI: 10.2140/gt.2019.23.2861. 


\bibitem{Vo17} Voisin, C. (2017). On the universal $CH_0$ group of cubic hypersurfaces. \textit{J. Eur. Math. Soc.} 19(6):1619--1653. DOI: 10.4171/JEMS/702. 

\end{thebibliography}
\end{document}